\documentclass[10pt]{amsart}

\usepackage{amsmath}
\usepackage{amssymb}
\usepackage{amsfonts}%
\usepackage{graphicx}
\usepackage{todonotes}
\usepackage{color}
\usepackage{mathtools}
\usepackage{bbm}
\usepackage{enumerate}
\usepackage[normalem]{ulem}

\newtheorem{theorem}{Theorem}[section]
\newtheorem{lemma}{Lemma}[section]
\theoremstyle{definition}

\theoremstyle{remark}
\newtheorem{remark}{Remark}[section]
\numberwithin{equation}{section}
\theoremstyle{plain}

\newtheorem{corollary}{Corollary}[section]

\newtheorem{proposition}{Proposition}[section]
\newtheorem{example}{Example}[section]

\begin{document}
\title[Primes in continued fraction expansions]{\textbf{Prime numbers in typical continued fraction expansions}}
\vspace{0.2cm}
\author{Tanja I. Schindler}
\address{Fakult\"{a}t f\"{u}r Mathematik, Universit\"{a}t Wien, Oskar-Morgenstern-Platz 1,
1090 Wien, Austria }
\email{tanja.schindler@univie.ac.at}

\author{Roland Zweim\"{u}ller}
\address{Fakult\"{a}t f\"{u}r Mathematik, Universit\"{a}t Wien, Oskar-Morgenstern-Platz 1,
1090 Wien, Austria }
\email{roland.zweimueller@univie.ac.at}

\thanks{\textit{Acknowledgement.} The authors are indebted to M.\ Thaler for valuable comments and suggestions on an
earlier version and to T.\ Trudgian for useful discussions regarding the error term of the prime number theorem.
This research was supported by the Austrian Science Fund FWF: P 33943-N}

\begin{abstract}
We study, from the viewpoint of metrical number theory and (infinite) ergodic
theory, the probabilistic laws governing the occurrence of prime numbers as
digits in continued fraction expansions of real numbers.
\end{abstract}
 \maketitle

\section{Introduction}

Ever since Gauss \cite{G} declared his interest in the intriguing statistical
properties of sequences of \emph{digits} $\mathsf{a}_{n}(x)$, $n\geq1$, in the
\emph{continued fraction} (\emph{CF}) \emph{expansion} of real numbers $x\in
I:=(0,1]$,
\[
x=\left[  \mathsf{a}_{1}(x),\mathsf{a}_{2}(x),\ldots\right]  =\frac
{1}{\mathsf{a}_{1}(x)+\dfrac{1}{\mathsf{a}_{2}(x)+\dfrac{1}{\mathsf{a}%
_{3}(x)+\cdots}}}%
\]
(and, in particular, mentioned that this led to questions he could not answer),
the \emph{metrical theory of continued fractions} has attracted many
mathematicians' attention. In the present paper we will be interested in the
\emph{prime digits} of $x$, i.e.\ those $\mathsf{a}_{n}(x)$ which happen to
belong to the set $\mathbb{P}$ of prime numbers. To single them out, we
define, for $x\in I$ and $n\geq1$,
\[
\mathsf{a}_{n}^{\prime}(x):=\mathbbm{1}_{\mathbb{P}}(\mathsf{a}_{n}(x))\cdot
\mathsf{a}_{n}(x)=\left\{
\begin{array}
[c]{ll}%
\mathsf{a}_{n}(x) & \text{if }\mathsf{a}_{n}(x)\in\mathbb{P}\text{,}\\
0 & \text{otherwise.}%
\end{array}
\right.
\]
(There is hardly any danger of misinterpreting this phonetically perfect
symbol as a derivative.) The purpose of this note is to point
out that it is in fact possible - with the aid of the prime number theorem
and recent work in (infinite) ergodic theory and in the probability theory
of dynamical systems - to derive a lot of information about the occurrences
and values of prime digits in CF-expansions of (Lebesgue-) typical numbers.
Besides stating the theorems themselves it is also our aim to show some newer more general results in ergodic theory in action. While many analogous versions of the following statements have directly been proven for the continued fraction digits, today, it is possible to deduce them or the version for the prime digits from more general theorems.

\section{Main Results - Pointwise matters}

We first consider questions about the pointwise behaviour of the sequence
$(\mathsf{a}_{n}^{\prime})_{n\geq1}$ on $I$. Throughout, $\lambda$ denotes
Lebesgue measure on (the Borel $\sigma$-field $\mathcal{B}_{I}$ of) $I$, and
\emph{almost everywhere }(\emph{a.e.}) is meant w.r.t.\ $\lambda$. For the sake
of completeness, we also include a few easy basic facts, e.g.\ that for a.e.\
$x\in I$, the proportion of those $k\in\{1,\ldots,n\}$ for which
$\mathsf{a}_{k}(x)$ is prime converges:

\begin{proposition}
[\textbf{Asymptotic frequency of prime digits}]\label{P_AsyFreqPrimes}We have
\begin{equation*}
\lim_{n\rightarrow\infty}\,\frac{1}{n}\sum_{k=1}^{n}\mathbbm{1}_{\mathbb{P}}%
\circ\mathsf{a}_{k}=\frac{1}{\log2}\,\log\prod_{\mathrm{p}\in\mathbb{P}%
}\left(  1+\frac{1}{\mathrm{p}(\mathrm{p}+2)}\right)  \text{ \quad a.e.}%
\end{equation*}
\end{proposition}%

\vspace{0.4cm}%

The results to follow can best be understood (and proved) by regarding
$(\mathsf{a}_{n}^{\prime})$ as a stationary sequence with respect to the Gauss
measure (cf. \S \ref{Sec_WarmUp} below). The first statement of the next theorem is parallel to the
classical Borel-Bernstein theorem (cf. \cite{Bo, Be})
the third and fourth statements are in accordance with \cite{KS}. As usual,
\emph{i.o.} is short for "infinitely often", i.e.\, "for infinitely many
indices". We denote the \emph{iterated logarithms} by$\ \log_{1}:=\log$ and
$\log_{m+1}:=\log\circ\log_{m}$, $m\geq1$.

Furthermore, we define the maximal entry $\mathsf{M}_{n}^{\prime}:=\max_{1\leq k\leq n}\mathsf{a}_{k}^{\prime}$, $n\geq1$.

\begin{theorem}
[\textbf{Pointwise growth of prime digits and maxima}]\label{T_PtwGrowth1}%
~

\noindent\textbf{a)} Assume that $(b_{n})_{n\geq1}$ is a sequence in $(1,\infty)$.
Then
\begin{equation}\label{eq: cond bn}
\lambda(\{\mathsf{a}_{n}^{\prime}>b_{n}\text{ i.o.}\})=\left\{
\begin{array}
[c]{ll}%
1 & \text{if }\sum_{n\geq1}\frac{1}{b_{n}\log b_{n}}=\infty\text{,}\\
0 & \text{otherwise.}%
\end{array}
\right.
\end{equation}
\vspace{0.4cm}
\noindent\textbf{b)} Moreover, if $(b_{n})_{n\geq1}$ is non-decreasing, then
\begin{equation}
\lambda(\{\mathsf{a}_{n}^{\prime}>b_{n}\text{ i.o.}\})=\lambda(\{\mathsf{M}%
_{n}^{\prime}>b_{n}\text{ i.o.}\})\text{. }\label{Eq_DigitvsMax}%
\end{equation}
\textbf{c)} Let $(c_n)_{n\geq1}$ and $(d_n)_{n\geq1}$ be sequences in $(1,\infty)$
with $d_n \to \infty$ and $c_n\leq d_n^{0.475}$ for large $n$. Then
\begin{equation*}
\lambda(\{d_n\leq \mathsf{a}_{n}^{\prime}\leq d_n (1+1/c_n)\text{ i.o.}\})
=\left\{
\begin{array}
[c]{ll}%
1 & \text{if }\sum_{n\geq1}\frac{1}{c_{n}d_n\log (d_n)}=\infty\text{.}\\
0 & \text{otherwise.}%
\end{array}
\right.
\end{equation*}
\textbf{d)} Let $(d_n)_{n\geq1}$ be a sequence of primes, then
\begin{equation*}
\lambda(\{\mathsf{a}_{n}^{\prime}=  d_n  \text{ i.o.}\})=\left\{
\begin{array}
[c]{ll}%
1 & \text{if }\sum_{n\geq1}\frac{1}{d_n^2}=\infty\text{,}\\
0 & \text{otherwise.}%
\end{array}
\right.
\end{equation*}
\end{theorem}

\begin{remark}
The exponent $0.475$ in c) comes from estimates for the error term in the prime number theorem and might be improved by
future research.
\end{remark}

\begin{example}
\label{Ex_PtwGrowth1}
A straightforward calculation shows that
\[
\lambda(\{\mathsf{a}_{n}^{\prime}>n\log_{2}^{\gamma}n\text{ i.o.}\})
=
\left\{
\begin{array}
[c]{ll}%
1 & \text{if }\gamma\leq1\text{,}\\
0 & \text{otherwise,}%
\end{array}
\right.
\]
and this remains true if $\mathsf{a}_{n}^{\prime}$ is replaced
by $\mathsf{M}_{n}^{\prime}$. We thus find that
\[
\underset{n\rightarrow\infty}{\overline{\lim}}\,\frac{\log\mathsf{a}%
_{n}^{\prime}-\log n}{\log_{3}n}=\,\underset{n\rightarrow\infty}%
{\overline{\lim}}\,\frac{\log\mathsf{M}_{n}^{\prime}-\log n}{\log_{3}%
n}=1\text{ \quad a.e.}%
\]
\end{example}%

As a consequence of Theorem \ref{T_PtwGrowth1} b), observing that the series
$\sum_{n\geq1}1/(b_{n}\log b_{n})$ converges iff $\sum_{n\geq1}1/(\rho
\,b_{n}\log(\rho\,b_{n}))$ converges for every $\rho\in(0,\infty)$, we get

\begin{corollary}\label{cor: Ex_PtwGrowth1}
If $(b_{n})_{n\geq1}$ is non-decreasing, then

 \begin{equation}
\underset{n\rightarrow\infty}{\overline{\lim}}\,\frac{\mathsf{a}_{n}^{\prime}%
}{b_{n}}=\,\underset{n\rightarrow\infty}{\overline{\lim}}\,\frac
{\mathsf{M}_{n}^{\prime}}{b_{n}}
=
\left\{
\begin{array}
[c]{ll}%
\infty \text{ \ a.e.} & \text{if }\sum_{n\geq1}\frac{1}{b_{n}\log b_{n}}=\infty\text{,}\\
0 \text{ \ a.e.} & \text{otherwise.}%
\end{array}
\right.
\end{equation}
In particular,
\begin{equation}
\underset{n\rightarrow\infty}{\overline{\lim}}\,
\frac{\mathsf{a}_{n}^{\prime}}{n\,\log_{2}n}=\,
\underset{n\rightarrow\infty}{\overline{\lim}}\,
\frac{\mathsf{M}_{n}^{\prime}}{n\,\log_{2}n}
=\infty\text{ \ a.e.}%
\end{equation}

\end{corollary}

\vspace{0.4cm}%

A convenient condition for the criterion above is provided by

\begin{lemma}
\label{L_SeriesLemma}Let $(b_{n})_{n\geq1}$ be a sequence in $(1,\infty)$ for
which $b_{n}/n$ increases. Then
\[
\underset{n\rightarrow\infty}{\overline{\lim}}\,\frac{n\log_{2}n}{b_{n}%
}>0\text{ \quad implies \quad}\sum_{n\geq1}\frac{1}{b_{n}\log b_{n}%
}=\infty\text{.}%
\]
\end{lemma}

As in the case of the full digit sequence $(\mathsf{a}_{n})_{n\geq1}$, the
peculiar properties of $(\mathsf{a}_{n}^{\prime})_{n\geq1}$ are due to the
fact that these functions are not integrable. A general fact for
non-integrable non-negative stationary sequences is the non-existence of a
non-trivial strong law of large numbers, made precise in a), c) and d) of the next
result, where c) is in the spirit of \cite{Phi88}. However, it is sometimes possible to recover a meaningful limit by
\emph{trimming}, i.e.\ by removing maximal terms. In the case of $(\mathsf{a}%
_{n})_{n\geq1}$, this was first pointed out in \cite{DV}. Assertion b) below
gives the proper version for the $(\mathsf{a}_{n}^{\prime})_{n\geq1}$.

\begin{theorem}
[\textbf{Strong laws of large numbers}]\label{T_SLLN}
~

\noindent\textbf{a)} The prime digits satisfy
\begin{equation*}
\lim_{n\rightarrow\infty}\,\frac{1}{n}\sum_{k=1}^{n}\mathsf{a}_{k}^{\prime
}=\infty\text{ \quad a.e.}\label{Eq_ETHMsimple}%
\end{equation*}

\vspace{0.4cm}
\noindent
\textbf{b)} Subtracting $\mathsf{M}_{n}^{\prime}$, we obtain a trimmed strong law,
\begin{equation}
\lim_{n\rightarrow\infty}\,\frac{\log2}{n\,\log_{2}n}\left(  \sum_{k=1}%
^{n}\mathsf{a}_{k}^{\prime}-\mathsf{M}_{n}^{\prime}\right)  =1\text{ \quad
a.e.}\label{Eq_TrimmedSLLN}%
\end{equation}%

\vspace{0.4cm}
\noindent
\textbf{c)}
For sequences $(b_{n})_{n\geq1}$ in $(1,\infty)$ satisfying
$b_{n}/n\nearrow\infty$ as $n\rightarrow\infty$,%
\begin{equation}
\underset{n\rightarrow\infty}{\overline{\lim}}\,\frac{1}{b_{n}}\sum_{k=1}%
^{n}\mathsf{a}_{k}^{\prime}=\infty\text{ \ a.e. \quad iff \quad}\sum_{n\geq
1}\frac{1}{b_{n}\log b_{n}}=\infty\text{,}\label{Eq_JonsETHM}%
\end{equation}
while otherwise
\begin{equation}\label{Eq_Dicho}
\lim_{n\rightarrow\infty}\,\frac{1}{b_{n}}\sum_{k=1}^{n}\mathsf{a}_{k}%
^{\prime}=0\text{ \ a.e.}%
\end{equation}%

\noindent\textbf{d) }
 But, defining $\overline{n}(j):=e^{j\log^{2}j}$, $j\geq1$, and
$d_{n}^{\prime}:=\overline{n}(j)\,\log_{2}\overline{n}(j)/\log2$ for
$n\in(\overline{n}(j-1),\overline{n}(j)]$ gives a normalizing sequence for
which
\begin{equation}
\underset{n\rightarrow\infty}{\overline{\lim}}\,\frac{1}{d_{n}^{\prime}}%
\sum_{k=1}^{n}\mathsf{a}_{k}^{\prime}=1\text{ \ a.e.}\label{Eq_SLLNfunnynorm}%
\end{equation}
\end{theorem}%

\vspace{0.4cm}%

The trimmed law from b) shows that the bad pointwise behaviour described in c)
is due to a few exceptionally large individual terms $\mathsf{a}_{n}^{\prime}$
which, necessarily, have to be of the order of the preceding partial sum
$\sum_{k=1}^{n-1}\mathsf{a}_{k}^{\prime}$. In fact, almost surely, the partial
sum will infinitely often be of strictly smaller order than the following
term, see statement a) below.
We can also ask whether, or to what extent, the
terms from the thinner sequence $(\mathsf{a}_{n}^{\prime})_{n\geq1}$ come
close to the partial sums $(\sum_{k=0}^{n-1}\mathsf{a}_{k})_{n\geq1}$ of the
unrestricted one. The answer is given by the dichotomy rule in statement b) of
the next result.

We shall tacitly interpret real sequences $(g_{n})_{n\geq0}$ as
functions on $\mathbb{R}_{+}$ via $t\longmapsto g_{[t]}$, and write $g(t)\sim
h(t)$ as $t\rightarrow\infty$ if $g(t)/h(t)\rightarrow1$.
Moreover, $g(t)\asymp h(t)$ means
$0<  {\underline{\lim}}_{t\rightarrow\infty}\, g(t)/h(t) \leq
{\overline{\lim}}_{t\rightarrow\infty}\,  g(t)/h(t)  < \infty$.

\begin{theorem}
[\textbf{Relative size of digits and partial sums}]\label{T_PDversusD1}%
~

\noindent\textbf{a)} We have
\begin{equation}
\underset{n\rightarrow\infty}{\overline{\lim}}\,\frac{\mathsf{a}_{n}^{\prime}%
}{\sum_{k=1}^{n-1}\mathsf{a}_{k}^{\prime}}=\,\infty\text{ \ a.e.}%
\label{Eq_DirektAusATZ}%
\end{equation}
Generally, for functions $g:[0,\infty)\rightarrow(3,\infty)$ fulfilling
$g(\eta(t))\asymp g(t)$ if $\eta(t) \sim t$ as $t\rightarrow\infty$,
we have%
\begin{equation}
\underset{n\rightarrow\infty}{\overline{\lim}}\,\frac{g(\mathsf{a}_{n}%
^{\prime})}{\sum_{k=1}^{n-1}\mathsf{a}_{k}^{\prime}}=\infty\text{ \ a.e. \quad
iff \quad}\int_{c}^{\infty}\frac{g(y)}{\log_{2}g(y)}\frac{dy}{y^{2}\log
y}=\infty\text{,}\label{Eq_ATZforStarStar}%
\end{equation}
while otherwise%
\begin{equation*}
\lim_{n\rightarrow\infty}\,\frac{g(\mathsf{a}_{n}^{\prime})}{\sum_{k=1}%
^{n-1}\mathsf{a}_{k}^{\prime}}=0\text{ \ a.e.}%
\end{equation*}%
\vspace{0.4cm}
\noindent
\textbf{b)} In contrast, comparing to the unrestricted digit sum $\sum
_{k=1}^{n-1}\mathsf{a}_{k}$, one has
\begin{equation*}
\lim_{n\rightarrow\infty}\,\frac{\mathsf{a}_{n}^{\prime}}{\sum_{k=1}%
^{n-1}\mathsf{a}_{k}}=0\text{ \ a.e.}
\end{equation*}
Generally, for functions $g:[0,\infty)\rightarrow(3,\infty)$ fulfilling
$g(\eta(t))\asymp g(t)$ if $\eta(t) \sim t$ as $t\rightarrow\infty$,
we have
\begin{equation*}
\underset{n\rightarrow\infty}{\overline{\lim}}\,\frac{g(\mathsf{a}_{n}%
^{\prime})}{\sum_{k=1}^{n-1}\mathsf{a}_{k}}=\infty\text{ \ a.e. \quad iff
\quad}\int_{c}^{\infty}\frac{g(y)}{\log g(y)}\frac{dy}{y^2\log y}%
=\infty\text{,}%
\end{equation*}
while otherwise
\begin{equation*}
\lim_{n\rightarrow\infty}\,\frac{g(\mathsf{a}_{n}^{\prime})}{\sum_{k=0}%
^{n-1}\mathsf{a}_{k}}=0\text{ \ a.e.}%
\end{equation*}%
\vspace{0.4cm}
\noindent
\textbf{c)} Turning to a comparison of partial sums, we find that
\begin{equation*}
\lim_{n\rightarrow\infty}\,\frac{\sum_{k=1}^{n}\mathsf{a}_{k}^{\prime}}%
{\sum_{k=1}^{n}\mathsf{a}_{k}}=0\text{ \ a.e.}
\end{equation*}
\end{theorem}

\vspace{0.4cm}

\begin{remark}
A broad class of functions which satisfy
$g(\eta(t))\asymp g(t)$ if $\eta(t) \sim t$ as $t\rightarrow\infty$,
are the regularly varying functions.
 Recall that a measurable function $g:(L,\infty)\rightarrow(0,\infty)$ is
\emph{regularly varying of index} $\rho\in\mathbb{R}$ at infinity, written
$g\in\mathcal{R}_{\rho}$, if $g(ct)/g(t)\rightarrow c^{\rho}$ as
$t\rightarrow\infty$ for all $c>0$ (see Chapter 1 of \cite{BGT} for more
information).
\end{remark}

Whether or not the integrals diverge can easily be checked for many specific
$g$'s:

\begin{example}
~

\noindent\textbf{a)} Taking $g(t):=t\,\log^{\rho}t$, $\rho\in\mathbb{R}$, part b)
gives
\[
\underset{n\rightarrow\infty}{\overline{\lim}}\,\frac{\mathsf{a}_{n}^{\prime
}\log\mathsf{a}_{n}^{\prime}}{\sum_{k=1}^{n-1}\mathsf{a}_{k}}=\infty\text{
\ a.e. for }\rho>1\quad\text{ while }\quad\lim_{n\rightarrow\infty}\,\frac{\mathsf{a}%
_{n}^{\prime}\log^{\rho}\mathsf{a}_{n}^{\prime}}{\sum_{k=1}^{n-1}%
\mathsf{a}_{k}}=0\text{ \ a.e.\ for }\rho\leq 1\text{.}%
\]
\vspace{0.4cm}
\textbf{b)} In case $g(t):=t\,\log t/\log_{2}^{\gamma}t$, $\gamma\in
\mathbb{R}$, we find for $\gamma\leq 1$
\[
\underset{n\rightarrow\infty}{\overline{\lim}}\,\frac{\mathsf{a}_{n}^{\prime
}\log\mathsf{a}_{n}^{\prime}/\log_{2}\mathsf{a}_{n}^{\prime}}{\sum_{k=1}%
^{n-1}\mathsf{a}_{k}}=\infty\text{ \ a.e.}%
\]
while, for $\gamma>1$,
\[
\lim_{n\rightarrow\infty}\,\frac{\mathsf{a}_{n}^{\prime}\log\mathsf{a}%
_{n}^{\prime}/\log_{2}^{\gamma}\mathsf{a}_{n}^{\prime}}{\sum_{k=1}%
^{n-1}\mathsf{a}_{k}}=0\text{ \ a.e.}%
\]
\end{example}%

On the other hand, if we look at primes to some power $\gamma$ we obtain - as a counterpart to Theorem \ref{T_SLLN} b) - the following result:

\begin{theorem}\label{thm: exponents of a}
~

\noindent\textbf{a)} For $\gamma<1$ there exists $K_\gamma>0$ such that
\begin{equation*}
\lim_{n\rightarrow\infty}\,\frac{\sum_{k=1}^{n}\left(\mathsf{a}_{k}^{\prime
}\right)^{\gamma}}{n}=K_\gamma <\infty\text{ \quad a.e.}
\end{equation*}

\vspace{0.4cm}%
%

\noindent\textbf{b)}
Let $\sigma:=\sigma_{(n,x)}\in\mathcal{S}_n$ be a pointwise permutation, i.e.\
$\sigma: I\times \{1,\ldots, n\} \to \{1,\ldots, n\}$,
such that
$\mathsf{a}_{\sigma(1)}^{\prime}\geq\ldots \geq \mathsf{a}_{\sigma(n)}^{\prime}$
and $\mathsf{S}_n^k:=\sum_{j=k+1}^{n}\mathsf{a}_{\sigma(j)}^{\prime}$.
If $\gamma>1$, then
for all $(b_n)\in\mathbb{N}^{\mathbb{N}}$ fulfilling
$b_n=o(n^{1-\epsilon})$ for some $\epsilon>0$ and
\begin{equation}
 \lim_{n\to\infty}\frac{b_n}{\log\log n}=\infty\label{eq: cond on bn}
\end{equation}
we have
\begin{equation}\label{eq: strong law gamma>1}
\lim_{n\rightarrow\infty}\,\frac{\mathsf{S}_n^{b_n}}{d_n}= 1\text{ \quad a.e.}
\end{equation}
where
\begin{equation}\label{eq: d_n}
 d_n\sim \frac{1}{(\gamma-1)(\log 2)^{\gamma}}\cdot \frac{n^{\gamma}b_n^{1-\gamma}}{(\log n)^{\gamma}}.
\end{equation}
\end{theorem}

\begin{remark}
 It is not proven that a trimming rate slower than the one given in \eqref{eq: cond on bn} is possible. However, by \cite{haeusler} one can deduce that for i.i.d.\ random variables with the same distribution function and $b_n\asymp \log\log n$ a strong law of large numbers as in \eqref{eq: strong law gamma>1} is no longer possible.
\end{remark}

However, if we only ask for convergence in probability, the picture looks much simpler and we refer the reader to Theorem \ref{thm: exponents of a in prob} in the next section.

\section{Main Results - Distributional matters}

The second set of results we present focuses on the distributions of (various
functions of) the digits $\mathsf{a}_{n}^{\prime}$. If $(M,d)$ is a separable metric
space with Borel $\sigma$-field $\mathcal{B}_{M}$, a sequence $(\nu
_{n})_{n\geq1}$ of probability measures on $(M,\mathcal{B}_{M})$
\emph{converges weakly} to the probability measure $\nu$ on $(M,\mathcal{B}%
_{M})$, written $\nu_{n}\Longrightarrow\nu$, if the integrals of bounded
continuous function $\psi:M\rightarrow\mathbb{R}$ converge, i.e.\ $\int
\psi\,d\nu_{n}\longrightarrow\int\psi\,d\nu$ as $n\rightarrow\infty$. If
$R_{n}:I\rightarrow M$, $n\geq 1$, Borel measurable functions and
$\nu$ a Borel probability on $M$ (or $R$ another \emph{random element} of $M$,
not necessarily defined on $I$, with distribution $\nu$) then $(R_{n}%
)_{n\geq1}$ \emph{converges in distribution} to $\nu$ (or to $R$) \emph{under
the probability measure} $P$ on $\mathcal{B}_{I}$, if the distributions
$P\circ R_{n}^{-1}$ of the $R_{n}$ w.r.t.\ $P$ converge weakly to $\nu$.
Explicitly specifying the underlying measure, we denote this by
\[
R_{n}\overset{P}{\Longrightarrow}\nu\text{ \quad or \quad}R_{n}\overset
{P}{\Longrightarrow}R\text{.}%
\]
For sequences $(R_{n})$ defined on an ergodic dynamical system, it is often
the case that a distributional limit theorem $R_{n}\overset{P}{\Longrightarrow
}R$ automatically carries over to a large collection of other probability
measures: \emph{strong distributional convergence}, written
\[
R_{n}\overset{\mathcal{L}(\lambda)}{\Longrightarrow}\nu\text{ \quad or \quad
}R_{n}\overset{\mathcal{L}(\lambda)}{\Longrightarrow}R\text{,}%
\]
means that $R_{n}\overset{P}{\Longrightarrow}R$ for all probability measures
$P\ll\lambda$, see \cite{Z7}.%

\vspace{0.4cm}%

We start by giving a counterpart to Theorem \ref{thm: exponents of a} for weak convergence, where b) is in the spirit of \cite{Khi}.

\begin{theorem}\label{thm: exponents of a in prob}
~

\noindent\textbf{a)} For $\gamma<1$ there exists $K_\gamma>0$ such that
\begin{equation*}
\frac{\sum_{k=1}^{n}\left(\mathsf{a}_{k}^{\prime
}\right)^{\gamma}}{n}\overset{\mathcal{L}(\lambda)}{\Longrightarrow
} K_\gamma.
\end{equation*}

\noindent\textbf{b)} For the case $\gamma=1$ we have
\begin{equation*}
\frac{\sum_{k=1}^{n}\mathsf{a}_{k}^{\prime
}}{n\log_2 n}\overset{\mathcal{L}(\lambda)}{\Longrightarrow
} \log 2.
\end{equation*}

\noindent\textbf{c)} If $\gamma>1$ we have
\begin{equation*}
\frac{\mathsf{S}_n^{b_n}}{d_n}\overset{\mathcal{L}(\lambda)}{\Longrightarrow
} 1,
\end{equation*}
where $\mathsf{S}_n^{b_n}$ is defined as in Theorem \ref{thm: exponents of a}, $(d_n)$ is given as in \eqref{eq: d_n} and $\lim_{n\to\infty}b_n=\infty$ and $b_n=o(n^{1-\epsilon})$.
\end{theorem}


\begin{remark}
 Indeed by \cite{KS_mean} the stronger result of convergence in mean follows for c).
 It is not proven that for the situation in c) convergence in probability can not hold for a lightly trimmed sum, i.e.\ a sum from which only a finite number of large entries, being independent of $n$ is removed.
 However, it follows from \cite{A2} that $\sum_{k=1}^n(\mathsf{a}_k^{\prime})^{\gamma}$ normed by the right norming sequence converges to a non-degenerate Mittag-Leffler distribution if $\gamma>1$.
 On the other hand, by \cite{kesten} it follows that light trimming does not have any influence on distributional convergence if the random variables considered are i.i.d.
\end{remark}

\vspace{0.4cm}

As we have seen in the previous section, the maximum $\mathsf{M}_{n}^{\prime}$ has a large influence then the whole system, in the following we will give its distributional convergence.
We let $\Theta$ denote a positive random variable with $\Pr[\Theta\leq y]=e^{-1/y}$, $y>0$ and get the following counterpart to \cite{Phi76}.

\begin{theorem}
[\textbf{Distributional convergence of} $\mathsf{M}_{n}^{\prime}$%
]\label{T_ForPhilipp}The maximum $\mathsf{M}_{n}^{\prime}$ of the prime digits
converges in distribution,
\begin{equation}
\frac{\log2\,\log n}{n}\cdot\mathsf{M}_{n}^{\prime}\overset{\mathcal{L}%
(\lambda)}{\Longrightarrow}\Theta\text{ }\quad\text{as }n\rightarrow
\infty\text{.}\label{Eq_StrgDistrCgeMax}%
\end{equation}

\end{theorem}%
\vspace{0.4cm}%

A related classical topic, introduced by Doeblin \cite{Doeblin}, is the Poissonian
nature of occurrences of very large CF-digits. For $l\geq 1$ let $\varphi
_{l}=\varphi _{l,1}:=\inf \{k\geq 1:\mathsf{a}_{k}\geq l\}$, the first
position in the CF-expansion at which a digit $\geq l$ shows up, and $%
\varphi _{l,i+1}:=\inf \{k\geq 1:\mathsf{a}_{\varphi _{l,i}+k}\geq l\}$ the
distance between the $i$th and $(i+1)$st occurrence. Defining $\Phi
_{l}:I\rightarrow \lbrack 0,\infty ]^{\mathbb{N}}$ as $\Phi _{l}:=(\varphi
_{l,1},\varphi _{l,2},\ldots )$ and letting $\Phi _{\mathrm{Exp}}$ denote an
i.i.d.\ sequence of normalized exponentially distributed random variables, we
can express this classical result by stating that
\begin{equation*}
\frac{1}{\log 2}\frac{1}{l}\cdot \Phi _{l}\overset{\lambda }{\Longrightarrow
}\,\Phi _{\mathrm{Exp}}\quad \text{as }l\rightarrow \infty .
\end{equation*}%
Turning to prime digits, we shall consider the corresponding quantities $%
\varphi _{l,i}^{\prime }$ with $\varphi _{l,0}^{\prime }:=0$ and $\varphi
_{l,i+1}^{\prime }:=\inf \{k\geq 1:\mathsf{a}_{\varphi _{l,i}^{\prime
}+k}^{\prime }\geq l\}$, $i\geq 0$, and the processes $\Phi _{l}^{\prime
}:=(\varphi _{l,1}^{\prime },\varphi _{l,2}^{\prime },\ldots )$ of distances
between consecutive occurrences of prime digits of size at least $l$. In
fact, we also provide refined versions of the limit theorem which show that,
asymptotically, both the relative size compared to $l$ of such a large prime
digit $\mathsf{a}_{\varphi _{l,i}^{\prime }}^{\prime }$ and its residue
class for a given modulus $m$, are stochastically independent of the
positions $\varphi _{l,i}^{\prime }$ at which they occur. (These statements
are parallel to Propositions 10.1 and 10.2 of \cite{ZHit}. A 
$(q_1,\ldots ,q_d)$-Bernoulli sequence is an iid sequence of random variables 
which can assume $d$ different values with respective probabilities $q_1,\ldots ,q_d$.)

\begin{theorem}[\textbf{Poisson limits for large prime CF-digits}]
\label{T_NewPoisson}The sequences $\Phi _{l}^{\prime }$ of positions at
which large prime digits occur satisfy the following.

\noindent\textbf{a)} Their distances converge to an i.i.d.\ sequence of exponential
variables,
\begin{equation}
\frac{1}{\log 2}\frac{1}{l\log l}\cdot \Phi _{l}^{\prime }\overset{\mathcal{L%
}(\lambda )}{\Longrightarrow }\,\Phi _{\mathrm{Exp}}\quad \text{as }%
l\rightarrow \infty .
\end{equation}

\noindent\textbf{b)} Take any $\vartheta \in (0,1)$, let $\psi _{l,i}^{\prime }\ $be
the indicator function of $\{\mathsf{a}_{\varphi _{l,i}^{\prime }}^{\prime
}\geq l/\vartheta \}$ and set $\Psi _{l}^{\prime }:=(\psi _{l,1}^{\prime
},\psi _{l,2}^{\prime },\ldots )$, which identifies those prime digits $\geq
l$ which are in fact $\geq l/\vartheta $. Then
\begin{equation}
\left( \frac{1}{\log 2}\frac{1}{l\log l}\cdot \Phi _{l}^{\prime },\Psi
_{l}^{\prime }\right) \overset{\mathcal{L}(\lambda )}{\Longrightarrow }%
(\,\Phi _{\mathrm{Exp}},\Psi ^{\prime })\quad \text{as }l\rightarrow \infty
\text{,}
\end{equation}%
where $(\,\Phi _{\mathrm{Exp}},\Psi ^{\prime })$ is an independent pair with
$\Psi ^{\prime }$ a $({ 1-\vartheta},\vartheta)$-Bernoulli sequence.

\vspace{0.4cm}

\noindent\textbf{c)} Fix an integer $m\geq 2$. For $l>m$ define $\upsilon
_{l,i}^{\prime }:I\rightarrow \{j\in \{1,\ldots ,m\}:j$ relatively prime to $%
m\}$ by $\upsilon _{l,i}^{\prime }(x):=j$ if
$\mathsf{a}_{\varphi_{l,i}^{\prime }}^{\prime }(x)\equiv j \mod m$,
so that $\Upsilon
_{l}^{\prime }:=(\upsilon _{l,1}^{\prime },\upsilon _{l,2}^{\prime },\ldots
) $ identifies the residue classes mod $m$ of the prime digits $\mathsf{a}%
_{\varphi _{l,i}^{\prime }}^{\prime }$. Then%
\begin{equation}
\left( \frac{1}{\log 2}\frac{1}{l\log l}\cdot \Phi _{l}^{\prime },\Upsilon
_{l}^{\prime }\right) \overset{\mathcal{L}(\lambda )}{\Longrightarrow }%
(\,\Phi _{\mathrm{Exp}},\Upsilon ^{\prime })\quad \text{as }l\rightarrow
\infty \text{,}
\end{equation}%
where $(\,\Phi _{\mathrm{Exp}},\Upsilon ^{\prime })$ is an independent pair
with $\Upsilon ^{\prime }$ a $(\frac{1}{\phi (m)},\ldots ,\frac{1}{\phi (m)}%
) $-Bernoulli sequence. (Here $\phi (m)$ denotes the Euler totient.)
\end{theorem}

\vspace{0.4cm}%

We finally look at the distribution of a function which counts how many $\mathsf{a}_n^{\prime}$ fall into particular sets $A_n$ giving a limit theorem in the spirit of \cite{Phi70, KS}.
We let $\mathcal{N}$ denote a positive random variable with $\Pr[\mathcal{N}\leq y]=\int_0^y e^{-t^2/2}\,\mathrm{d}t/\sqrt{2\pi}$, $y>0$.

\begin{theorem}[\textbf{A CLT for counting primes in CF}]\label{thm: counting CLT}
Suppose that either
\begin{enumerate}[\rm (A)]
\item\label{en: clt 1} $A_n\coloneqq \left\{\mathsf{a}_n^{\prime}\geq b_n\right\}$ with $(b_n)\in\mathbb{R}^{\mathbb{N}}$ and $\sum_{n:b_{n}>1} {1}/{b_n\log b_n} =\infty$,
\item\label{en: clt 2} $A_n\coloneqq {\left\{\mathsf{a}_n^{\prime}= d_n\right\}}$ with $(d_n)$ a sequence of primes and $\sum_{n\in\mathbb{N}}{1}/{d_{n}^2}=\infty$,
 \item\label{en: clt 3} $A_n\coloneqq \left\{d_n\leq \mathsf{a}_n^{\prime} \leq d_n \cdot \left(1+\frac{1}{c_n}\right)\right\} $
with $(d_n)$ a sequence of natural numbers tending to infinity, $(c_n)$ a sequence
of positive numbers with $c_n \leq d_n^{0.475}$ and
$\sum_{n=1}^{\infty}{1}/{\left(c_n d_n\log(d_n)\right)}=\infty$.

\end{enumerate}
Then, for $S_n\coloneqq \sum_{k=1}^n\mathbbm{1}_{A_k}$ the following central limit theorem holds:
\begin{equation*}
\frac{S_n-\int S_n\,\mathrm{d}\mu_{\mathfrak{G}}}{
 \sqrt{ \int \left(S_n-\int S_n\,\mathrm{d}\mu_{\mathfrak{G}}\right)^2}}\overset{\mathcal{L}%
(\lambda)}{\Longrightarrow}\mathcal{N}\text{ }\quad\text{as }n\rightarrow
\infty.
\end{equation*}
\end{theorem}

\section{The Gauss map and the prime digit
function\label{Sec_WarmUp}}

The results announced above express properties of certain stochastic processes
derived from the exceptionally well understood dynamical system generated by
the ergodic \emph{continued fraction map} (or \emph{Gauss map})
\[
S:(0,1]\rightarrow\lbrack0,1]\text{, \quad}Sx:=\frac{1}{x}-\left\lfloor
\frac{1}{x}\right\rfloor =\frac{1}{x}-k\text{ for }x\in\left(  \frac{1}%
{k+1},\frac{1}{k}\right]  =:I_{k}\text{, }k\geq1
\]
which, since \cite{G}, is known to preserve the probability density%
\[
h_{\mathfrak{G}}(x):=\frac{1}{\log2}\frac{1}{1+x}\text{, \quad}x\in I\text{.}%
\]
The invariant \emph{Gauss measure} $\mu_{\mathfrak{G}}$ on $\mathcal{B}_{I}$
defined by the latter, $\mu_{\mathfrak{G}}(B):=\int_{B}h_{\mathfrak{G}%
}(x)\,dx$, is exact (and hence ergodic). As hardly any textbook on ergodic
theory fails to point out, iteration of $S$ reveals the continued fraction
digits of any $x\in I$, in that
\[
x=\left[  \mathsf{a}_{1}(x),\mathsf{a}_{2}(x),\ldots\right]  \text{ \quad with
\quad}\mathsf{a}_{n}(x)=\mathsf{a}\circ S^{n-1}(x)\text{, }n\geq1\text{,}%
\]
where $\mathsf{a}:I\rightarrow\mathbb{N}$ is the \emph{digit function}
corresponding to the partition $\xi:=\{I_{k}:k\geq1\}$, i.e.\ $\mathsf{a}%
(x):=\left\lfloor 1/x\right\rfloor =k$ for $x\in I_{k}$. The stationary
sequence $(\mathsf{a}\circ S^{n})_{n\geq0}$ on the probability space
$(I,\mathcal{B}_{I},\mu_{\mathfrak{G}})$ thus obtained exhibits interesting
properties since $\mathsf{a}$ has infinite expectation, $\int_{I}%
\mathsf{a}\,d\mu_{\mathfrak{G}}=\sum_{k\geq1}k\,\mu_{\mathfrak{G}}%
(I_{k})=\infty$, as $\mu_{\mathfrak{G}}(I_{k})=\log(\frac{(k+1)^{2}}%
{k(k+2)})/\log2\sim1/(\log2\cdot k^{2})$ for $k\rightarrow\infty$. As in
classical probability theory, the \emph{tail behaviour of the distribution},
given by
\begin{equation*}
\mu_{\mathfrak{G}}\left(  \left\{  \mathsf{a}\geq K\right\}  \right)
=\frac{1}{\log2}\cdot\log\left(  \frac{K+1}{K}\right)  \sim\frac{1}{\log
2}\cdot\frac{1}{K}\text{ \quad as }K\rightarrow\infty
\end{equation*}
(which entails $L(N):=\int_{I}(\mathsf{a}\wedge N)\,d\mu_{\mathfrak{G}}%
=\sum_{K=1}^{N}\mu_{\mathfrak{G}}\left(  \left\{  \mathsf{a}\geq K\right\}
\right)  \sim\log N/\log2$ as $N\rightarrow\infty$), is the key to fine
asymptotic results. However, the study of the CF digit sequence goes beyond
standard results, since the random variables $\mathsf{a}\circ S^{n}$ are not
independent. Yet, it is well known that they still satisfy a strong form of
\emph{asymptotic independence} or \emph{mixing} in the following sense:

Given any measure preserving transformation $T$ on a probability space
$(\Omega,\mathcal{B},P)$, and a countable measurable partition $\gamma$ (mod
$P$), the $\psi$\emph{-mixing coefficients of} $\gamma$ are defined as
\[
\psi_{\gamma}(n):=\sup_{k\geq1}\left\{  \left\vert \frac{P(V\cap
W)}{P(V)P(W)}-1\right\vert :%
\begin{array}
[c]{ll}%
V\in\sigma(\bigvee_{j=0}^{k-1}T^{-j}\gamma),P(V)>0, & \\
W\in T^{-(n+k-1)}\mathcal{B},P(W)>0 &
\end{array}
\right\}  \text{, }n\geq1\text{.}%
\]

The partition $\gamma$ is said to be \emph{continued-fraction (CF-) mixing}
for the probability preserving system $(\Omega,\mathcal{B},P,T)$ if it is
generating, and if $\psi_{\gamma}(1)<\infty$ as well as $\psi_{\gamma
}(n)\rightarrow0$ for $n\rightarrow\infty$. (Note that $(\psi_{\gamma
}(n))_{n\geq1}$ is non-increasing.) Of course, the nomenclature is due to the
fact that
\begin{equation}
\xi\text{ is CF-mixing for }(I,\mathcal{B}_{I},\mu_{\mathfrak{G}}%
,S)\text{.}\label{Eq_CFmixforCF}%
\end{equation}
Actually, this system is \emph{exponentially CF-mixing}, in that there are
constants $C>0$ and $\rho\in(0,1)$ such that
\begin{equation*}
\psi_{\xi}(n)\leq C\,\rho^{n}\text{ \quad for }n\geq1
\end{equation*}
(which is related to Gauss' famous question mentioned in the introduction, see
e.g.\ \cite{IK} or \cite{Z4}).

\vspace{0.4cm}%

We are going to study occurrences of \emph{prime digits} by considering the
\emph{restricted digit function}
$\mathsf{a}^{\prime}:=
(\mathbbm{1}_{\mathbb{P}} \circ \mathsf{a}) \cdot \mathsf{a}
:I\rightarrow\{0\}\cup\mathbb{P}$.
As in
the case of $\mathsf{a}$, this function, as a random variable on
$(I,\mathcal{B}_{I},\mu_{\mathfrak{G}})$, still has infinite expectation.
Indeed, the prime number theorem (PNT) enables us to quickly determine the
all-important tail asymptotics for the distribution of $\mathsf{a}^{\prime}$.
The following lemma is the key to our analysis of the prime digit sequence.

\begin{lemma}
[\textbf{Tail behaviour and truncated expectation of} $\mathsf{a}^{\prime}$%
]\label{L_PrimeDigitTail}The distribution of $\mathsf{a}^{\prime}$ (with
respect to the Gauss measure) satisfies
\begin{equation}
\mu_{\mathfrak{G}}\left(  \left\{  \mathsf{a}^{\prime}\geq K\right\}  \right)
\sim\frac{1}{\log2}\cdot\frac{1}{K\log K}\text{ \quad as }K\rightarrow
\infty\text{.}\label{Eq_PrimeDigitTail}%
\end{equation}
In particular, $\mathsf{a}^{\prime}$ is not integrable, $\int_{I}%
\mathsf{a}^{\prime}\,d\mu_{\mathfrak{G}}=\infty$. Moreover,
\begin{equation}
L^{\prime}(N):=\int_{I}(\mathsf{a}^{\prime}\wedge N)\,d\mu_{\mathfrak{G}}%
\sim\frac{\log_{2}N}{\log2}\text{ \quad as }N\rightarrow\infty\text{.}%
\label{Eq_PrimeDigitWR}%
\end{equation}
so that $a^{\prime}(N):=N/L^{\prime}(N)\sim\log2\cdot N/\log_{2}N$ is
asymptotically inverse to $b^{\prime}(N):=(N\,\log_{2}N)/\log2$.
\end{lemma}
\vspace{0.4cm}

\begin{proof}
First, the PNT is easily seen (cf. \cite{HW}, Theorem 1.8.8) to imply that
\begin{equation}
\mathrm{p}_{n}\sim n\log n\text{ \quad as }n\rightarrow\infty\text{,}%
\label{Eq_PNT2}%
\end{equation}
where $\mathrm{p}_{n}$ denotes the $n$th prime number. Therefore,
\[
\sum_{n\geq N}\frac{1}{\mathrm{p}_{n}^{2}}\sim\sum_{n\geq N}\frac{1}%
{n^{2}(\log n)^{2}}\sim\frac{1}{N(\log N)^{2}}\text{ \quad as }N\rightarrow
\infty\text{.}%
\]
Letting $N(K)$ denote the least $n$ with $\mathrm{p}_{n}\geq K$, we have, as
$K\rightarrow\infty$,
\[
\mu_{\mathfrak{G}}\left(  \left\{  \mathsf{a}^{\prime}\geq K\right\}  \right)
=\sum_{\mathrm{p}\geq K,\mathrm{p}\in\mathbb{P}}\mu_{\mathfrak{G}%
}(I_{\mathrm{p}})\sim\frac{1}{\log2}\sum_{\mathrm{p}\geq K,\mathrm{p}%
\in\mathbb{P}}\frac{1}{\mathrm{p}^{2}}=\frac{1}{\log2}\sum_{n\geq N(K)}%
\frac{1}{\mathrm{p}_{n}^{2}}
\]
and, by PNT, $N(K)\sim K/\log K$. Combining these observations yields
(\ref{Eq_PrimeDigitTail}). The second statement is an easy consequence
thereof, since
\[
L^{\prime}(N)=\sum_{K=1}^{N}\mu_{\mathfrak{G}}\left(  \left\{  \mathsf{a}%
^{\prime}\geq K\right\}  \right)  \sim\frac{1}{\log2}\sum_{K=2}^{N}\frac
{1}{K\log K}\sim\frac{1}{\log2}\int_{2}^{N}\frac{dx}{x\log x}\text{ \ as
}N\rightarrow\infty\text{.}%
\]
Straightforward calculation verifies the assertions about $a^{\prime}$ and
$b^{\prime}$.
\end{proof}%

\begin{remark}
Several of the results allow for analogues in which prime digits are
replaced by digits belonging to other subsets $\mathbb{M}$ of the integers
for which $\pi _{\mathbb{M}}(n):=\#\mathbb{M}\cap \{1,\ldots ,n\}$ is
regularly varying with $\sum_{m\in \mathbb{M}}\frac{1}{m}=\infty $, like,
for example, the set of integers which are the product of exactly $k$ prime
numbers, see Theorem 3.5.11 of \cite{Ja}. (M.~Thaler, personal communication.)
\end{remark}

\section{Proofs of the results on a.e.\ convergence}

We are now ready for the proofs of our pointwise convergence results. We can
always work, without further mention, with the invariant measure
$\mu_{\mathfrak{G}}$, since it has the same null-sets as $\lambda$.

\begin{proof}
[\textbf{Proof of Proposition \ref{P_AsyFreqPrimes}.}] This, of course, is just
the ergodic theorem,
\[
\frac{1}{n}\sum_{k=1}^{n}\mathbbm{1}_{\mathbb{P}}\circ\mathsf{a}_{k}=\frac{1}{n}%
\sum_{k=0}^{n-1}\mathbbm{1}_{\mathbb{P}}\circ S^{k}\longrightarrow\mu_{\mathfrak{G}}%
(\mathbb{P})=\sum_{\mathrm{p}\in\mathbb{P}}\mu_{\mathfrak{G}}(I_{\mathrm{p}})\text{
\quad a.e.\ as }n\rightarrow\infty\text{.}%
\]
\end{proof}%

\vspace{0.4cm}%

In the following we will repeatedly appeal to the following version of
R\'{e}nyi's Borel-Cantelli Lemma (BCL) (as in Lemma 1 of \cite{ATZ}):
\begin{lemma}
[\textbf{R\'{e}nyi's Borel-Cantelli Lemma}]\label{L_RBCL}Assume that
$(E_{n})_{n\geq1}$ is a sequence of events in the probability space
$(\Omega,\mathcal{B},P)$ for which there is some $r\in(0,\infty)$ such that
\begin{equation*}
\frac{P(E_{j}\cap E_{k})}{P(E_{j})\,P(E_{k})}\leq r\qquad\text{whenever
}j,k\geq1\text{, }j\neq k\text{.}
\end{equation*}
Then $P(\{E_{n}$ infinitely often$\})>0$ iff $\sum_{n\geq1}P(E_{n})=\infty$.
\end{lemma}
\vspace{0.4cm}

This lemma enables us to prove Theorem \ref{T_PtwGrowth1}.\\

\begin{proof}
[\textbf{Proof of Theorem \ref{T_PtwGrowth1}.}]

\noindent\textbf{a)}
Note that
$\{\mathsf{a}_{j}^{\prime}>c\}=S^{-(j-1)}\{\mathsf{a}^{\prime}>c\}$ with
$\{\mathsf{a}^{\prime}>c\}$ measurable w.r.t.\ $\xi$. As a consequence of the
CF-mixing property (\ref{Eq_CFmixforCF}), we see that R\'{e}nyi's BCL applies
to show that
\begin{equation}
\mu_{\mathfrak{G}}(\{\mathsf{a}_{n}^{\prime}>b_{n}\text{ i.o.}\})>0\text{
\quad iff \quad}\sum_{n\geq1}\mu_{\mathfrak{G}}(\{\mathsf{a}_{n}^{\prime
}>b_{n}\})=\infty\text{.}\label{Eq_RBCLa}%
\end{equation}
By $S$-invariance of $\mu_{\mathfrak{G}}$ and Lemma \ref{L_PrimeDigitTail}, we
have $\mu_{\mathfrak{G}}(\{\mathsf{a}_{n}^{\prime}\geq b_{n}\})=\mu
_{\mathfrak{G}}(\{\mathsf{a}^{\prime}\geq b_{n}\})\sim1/(b_{n}\log b_{n})$, so
that divergence of the right-hand series in (\ref{Eq_RBCLa}) is equivalent to
that of $\sum_{n\geq1}(b_{n}\log b_{n})^{-1}$. Finally, again because of
$\{\mathsf{a}_{j}^{\prime}>c\}=S^{-(j-1)}\{\mathsf{a}^{\prime}>c\}$, the
set$\ \{\mathsf{a}_{n}^{\prime}>b_{n}$ i.o.$\}$ is easily seen to belong to
the tail-$\sigma$-field $\bigcap_{n\geq0}S^{-n}\mathcal{B}_{I} $ of $S$. The
system $(I,\mathcal{B}_{I},\mu_{\mathfrak{G}},S)$ being exact, the latter is
trivial mod $\mu_{\mathfrak{G}}$. Hence $\mu_{\mathfrak{G}}(\{\mathsf{a}%
_{n}^{\prime}>b_{n}$ i.o.$\})>0$ implies $\mu_{\mathfrak{G}}(\{\mathsf{a}%
_{n}^{\prime}>b_{n}$ i.o.$\})=1$.

\vspace{0.4cm}%
%

\noindent
\textbf{b) }Statement (\ref{Eq_DigitvsMax}) is seen by an easy routine
argument, as in the proof of Proposition 3.1.8 of \cite{IK}.

\vspace{0.4cm}
\noindent\textbf{c) }Without loss of generality we first assume that $c_n\leq 0.5$, for all $n$. If this doesn't hold, we can easily switch to a subsequence in which this holds and consider the subsequences separately.
By the prime number theorem we have
\begin{align*}
 \lambda(\mathsf{a}_{1}^{\prime}\in [d_n, d_n(1+1/c_n)])
 &= \lambda(\mathsf{a}_{1}^{\prime}\geq d_n)
 -\lambda(\mathsf{a}_{1}^{\prime}\geq d_n(1+1/c_n))\\
 &\asymp \frac{1}{d_n\log d_n}-\frac{1}{d_n(1+1/c_n)\log (d_n(1+1/c_n))}\\
 &\sim \frac{1}{c_nd_n\log d_n}.
\end{align*}

Next, we assume that $c_n>0.5$. We note that
\begin{align}
\MoveEqLeft\#\mathbb{P}\cap (d_n, d_n(1+1/c_n)]\cdot \lambda(\mathsf{a}_{1}=d_n)\notag\\
 &\leq \lambda(\mathsf{a}_{1}^{\prime}\in [d_n, d_n(1+1/c_n)])\notag\\
 &\leq \#\mathbb{P}\cap (d_n, d_n(1+1/c_n)]\cdot \lambda(\mathsf{a}_{1}=d_n(1+1/c_n)).\label{eq: primes times lambda}
\end{align}
Furthermore,
\begin{align}
 \lambda(\mathsf{a}_{1}=d_n)\asymp \frac{1}{d_n^2}\asymp \lambda(\mathsf{a}_{1}=d_n(1+1/c_n)).\label{eq: primes times lambda1}
\end{align}

On the other hand, we have by \cite{BHP_primes}, p.\ 562 that there exists $K>0$ such that
\begin{align*}
 \#\mathbb{P}\cap (d_n, d_n(1+1/c_n)]
 &\sim \pi (d_n(1+1/c_n))-\pi(d_n)
 \leq K\cdot \frac{d_n}{c_n\log d_n}.
\end{align*}
Combining this with \eqref{eq: primes times lambda} and \eqref{eq: primes times lambda1} yields the statement of c).

\vspace{0.4cm}%

\noindent
\textbf{d) }
This follows immediately from \cite[Theorem 6a]{KS}.
\end{proof}%

\vspace{0.4cm}%

\begin{proof}[Proof of Lemma \ref{L_SeriesLemma}]

By assumption there is some $\varepsilon \in (0,1)$ such that the set
$ M:= \{ n \geq 1: (n \log_2 n)/b_n \geq \varepsilon \} $
is infinite. Define $c(x):=\exp(\sqrt{\log x})$ and
$f(x):=x \log x \log_2 x$ for $x>1$, and note that
$c(x)<x$ for $x>e$.

Suppose that $n\in M$, $n \geq 4$, and $c(n)<k\leq n$.
Since $k\leq n$, we have $b_k = (b_k/k)k \leq (b_n/n)k
\leq (1/\varepsilon) k \log_2 n$, and thus
\begin{align*}
b_k \log b_k
& \leq
(1/\varepsilon) k \log_2 n (\log(1/\varepsilon) +\log k + \log_3 n)\\
& =
\frac{f(k)}{\varepsilon} \frac{\log_2 n}{\log_2 k}
\left(-
\frac{\log (\varepsilon)}{ \log k} + 1 + \frac{\log_3 n}{\log k}
\right).
\end{align*}
On the other hand, $c(n)<k$ implies $\log k > \sqrt{\log n}$ and hence
$$
\log_2 k > (1/2) \log_2 n,\;
\log k > 1, \;\text{ and }\; \log k > \log_3 n.
$$
Using these estimates we see that
$$
b_k \log b_k \leq  \frac{f(k)}{C(\varepsilon)} \;\;\;\text{ with } \;\;\;
C(\varepsilon):= \frac{\varepsilon}{2 \left( -\log (\varepsilon)+2  \right)} >0.
$$
Taking into account that $\log_3 x$ is a primitive of $1/f(x)$ we get
\begin{align*}
\sum_{k>c(n)} \frac{1}{b_k \log b_k}
& \geq
C(\varepsilon)   \sum_{c(n)<k\leq n} \frac{1}{f(k)}\\
& \geq
C(\varepsilon) \left( \int_{c(n)}^{n} \frac{dx}{f(x)}
 - \frac{1}{f( \left[  c(n) \right]   )} \right)
 =
C(\varepsilon) \left( \log 2
 - \frac{1}{f( \left[  c(n) \right]   )} \right).
\end{align*}
Since this estimate holds for infinitely many $n$, we see that
$$
\underset{n\rightarrow\infty}{\overline{\lim}}\,
\sum_{k>n} \frac{1}{b_k \log b_k}
\geq  C(\varepsilon)  \log 2,
$$
proving that $\sum_{k \geq 1} \frac{1}{b_k \log b_k}$ diverges.
\end{proof}%

\vspace{0.4cm}%

\begin{proof}
[\textbf{Proof of Theorem \ref{T_SLLN}.}]
~

\noindent\textbf{a)}
Since $\int_{I}%
\mathsf{a}^{\prime}\,d\mu_{\mathfrak{G}}=\infty$ by Lemma
\ref{L_PrimeDigitTail}, this is immediate from the ergodic theorem.

\vspace{0.4cm}%

\noindent
\textbf{b) }We apply Theorem 1.1 of \cite{AN} to $(I,\mathcal{B}_{I}%
,\mu_{\mathfrak{G}},S)$ and $\mathsf{a}^{\prime}$, observing that (in the
notation of that paper), $\mathfrak{N}_{\mathsf{a}^{\prime}}=1$ since
$J_{1}=\sum_{n\geq1}\left(  n^{2}\,\log n\,\log_{2}n\right)  ^{-1}<\infty$.
Furthermore, by using the estimate of Lemma \ref{L_PrimeDigitTail} and setting $a^{\prime}(N):=N/L^{\prime}(N)\sim\log2\cdot N/\log_{2}N$ we get that its
asymptotic inverse can be written as $b^{\prime}(N):=(N\,\log_{2}N)/\log2$ which by the statement of the paper coincides with the norming sequence.

\vspace{0.4cm}%

\noindent \textbf{c) }
Using Theorem \ref{T_PtwGrowth1} a), we first note that
$\sum_{n\geq1}1/(b_{n}\log b_{n})=\infty$ implies $\overline{\lim
}_{n\rightarrow\infty}\,b_{n}^{-1}\sum_{k=1}^{n}\mathsf{a}_{k}^{\prime}%
=\infty$\ a.e. since $\mathsf{a}_{n}^{\prime}\leq\sum_{k=1}^{n}\mathsf{a}%
_{k}^{\prime}$.

For the converse, assume that $\sum_{n\geq1}1/(b_{n}\log b_{n})<\infty$,
which by Lemma \ref{L_SeriesLemma} entails $ (n \log_2 n)/b_n \to 0$. In view of Theorem \ref{T_PtwGrowth1} a),
our assumption implies that $\mathsf{M}_{n}^{\prime}/b_n \to 0$ a.e.
Together with statement b) above, these observations prove (\ref{Eq_Dicho}), because
\[
\frac{1}{b_{n}}\sum_{k=1}^{n}\mathsf{a}_{k}^{\prime}=
\frac{\mathsf{M}_{n}^{\prime}}{b_{n}}+
\frac{n\log_{2}n}{\log2\cdot b_{n}}\cdot
\frac{\log2}{n\log_{2}n}
\left(
\sum_{k=1}^{n}\mathsf{a}_{k}^{\prime} - \mathsf{M}_{n}^{\prime}
\right)
\text{.}%
\]

\vspace{0.4cm}%

\noindent
\textbf{d) } Note first that letting $c_{n}^{\prime}:=d_{n}^{\prime}/\log
_{2}(10j)$ for $n\in(\overline{n}(j-1),\overline{n}(j)]$, provides us with a
non-decreasing sequence satisfying $\sum_{n\geq1}1/(c_{n}^{\prime}\log
c_{n}^{\prime})<\infty$ (use generous estimates). By Theorem
\ref{T_PtwGrowth1} therefore $\lambda(\{\mathsf{M}_{n}^{\prime}>c_{n}^{\prime}
$ i.o.$\})=0$. Since $c_{n}^{\prime}=o(d_{n}^{\prime})$, we see that for every
$\varepsilon>0$, $\{\mathsf{M}_{n}^{\prime}>\varepsilon\,d_{n}^{\prime}$
i.o.$\}\subseteq\{\mathsf{M}_{n}^{\prime}>c_{n}^{\prime}$ i.o.$\}$. Combining
these observations shows that
\begin{equation}
\frac{\mathsf{M}_{n}^{\prime}}{d_{n}^{\prime}}\longrightarrow0\text{ \quad
a.e.}\label{Eq_Rust}%
\end{equation}
Together with (\ref{Eq_TrimmedSLLN}) and $n\log_{ 2}  n/(\log2\cdot d_{n}^{\prime
})\leq1$, this proves, via
\begin{equation}
\frac{1}{d_{n}^{\prime}}\sum_{k=1}^{n}\mathsf{a}_{k}^{\prime}=\frac{n\log
_{2}n}{\log2\cdot d_{n}^{\prime}}\cdot\frac{\log2}{n\log_{2}n}\left(
\sum_{k=1}^{n}\mathsf{a}_{k}^{\prime}-\mathsf{M}_{n}^{\prime}\right)
+\frac{\mathsf{M}_{n}^{\prime}}{d_{n}^{\prime}}\text{,}\label{Eq_Rudi}%
\end{equation}
that
\[
\underset{n\rightarrow\infty}{\overline{\lim}}\,\frac{1}{d_{n}^{\prime}}%
\sum_{k=1}^{n}\mathsf{a}_{k}^{\prime}\leq1\text{ \ a.e.}%
\]
Specializing \eqref{Eq_Rudi}, and using \eqref{Eq_TrimmedSLLN} and
\eqref{Eq_Rust} again, we find that
\[
\frac{1}{d_{\overline{n}(j)}^{\prime}}\sum_{k=1}^{\overline{n}(j)}%
\mathsf{a}_{k}^{\prime}=\frac{\log2}{\overline{n}(j)\log_{2}\overline{n}%
(j)}\left(  \sum_{k=1}^{\overline{n}(j)}\mathsf{a}_{k}^{\prime}-\mathsf{M}%
_{\overline{n}(j)}^{\prime}\right)  +\frac{\mathsf{M}_{\overline{n}%
(j)}^{\prime}}{d_{\overline{n}(j)}^{\prime}}\longrightarrow1\text{ \quad a.e.}%
\]
as $j\rightarrow\infty$, and our claim (\ref{Eq_SLLNfunnynorm}) follows.
\end{proof}%

\vspace{0.4cm}%

\begin{proof}
[\textbf{Proof of Theorem \ref{T_PDversusD1}.}]
~

\noindent\textbf{a)}
Apply Theorem 4 of
\cite{ATZ} to the system $(I,\mathcal{B}_{I},\mu_{\mathfrak{G}},S)$ with
CF-mixing partition $\gamma:=\xi$. Statement (\ref{Eq_DirektAusATZ}) is
immediate if we take $(\mathsf{a}^{\prime},\mathsf{a}^{\prime})$ as our pair
$(\varphi,\psi)$ of $\gamma$-measurable functions, cf. Remark 3 in \cite{ATZ}.
Turning to the general version \eqref{Eq_ATZforStarStar}, we consider
$\varphi:=g\circ\mathsf{a}^{\prime}$ and $\psi:=\mathsf{a}^{\prime}$.
According to the result cited,
\[
\underset{n\rightarrow\infty}{\overline{\lim}}\,\frac{g(\mathsf{a}_{n}%
^{\prime})}{\sum_{k=1}^{n-1}\mathsf{a}_{k}^{\prime}}=\infty\text{ \ a.e. \quad
iff \quad}\int_{I}a^{\prime}\circ g\circ\mathsf{a}^{\prime}\,d\mu
_{\mathfrak{G}}=\infty\text{, }%
\]
(with $a^{\prime}$ from Lemma \ref{L_PrimeDigitTail}), while otherwise
$\lim_{n\rightarrow\infty}\,g(\mathsf{a}_{n}^{\prime})/(\sum_{k=1}%
^{n-1}\mathsf{a}_{k}^{\prime})=0$ \ a.e.
The present assertion merely
reformulates the divergence condition above:
We see (using \eqref{Eq_PNT2} and the regularity properties on $g$) that (for some constant $c>0$)
\begin{align*}
\int_{I}a\circ g\circ\mathsf{a}^{\prime}\,d\mu_{\mathfrak{G}}  & \asymp
\sum_{n\geq1}\frac{g(\mathrm{p}_{n})}{\log_{2}g(\mathrm{p}_{n})}%
\,\mu_{\mathfrak{G}}(I_{\mathrm{p}_{n}})\asymp\sum_{n\geq1}\frac{g(n\log
n)}{\log_{2}g(n\log n)}\frac{1}{(n\log n)^{2}}\\
& \asymp\int_{c}^{\infty}\frac{g(x\log x)}{\log_{2}g(x\log x)}\frac{dx}{(x\log
x)^{2}}\asymp\int_{c}^{\infty}\frac{g(y)}{\log_{2}g(y)}\frac{dy}{y^{2}\log
y}\text{.}%
\end{align*}%

\vspace{0.4cm}%

\noindent
\textbf{b)} Same argument as in a), this time with $\varphi:=g\circ
\mathsf{a}^{\prime}$ and $\psi:=\mathsf{a}$, and replacing $a^{\prime}$ above
by $a(t):=t/L(t)\sim\log2\cdot t/\log t$ as $t\rightarrow\infty$.

\vspace{0.4cm}%

\noindent
\textbf{c)}
We have
\begin{align*}
 \lim_{n\to\infty}\frac{\sum_{k=1}^n \mathsf{a}_k^{\prime}}{\sum_{k=1}^n \mathsf{a}_k}
 &\leq \lim_{n\to\infty}\frac{\sum_{k=1}^n \mathsf{a}_k^{\prime}-\mathsf{M}_n^{\prime}}{\sum_{k=1}^n \mathsf{a}_k-\mathsf{M}_n}+\lim_{n\to\infty}\frac{\mathsf{M}_n^{\prime}}{\sum_{k=1}^n \mathsf{a}_k-\mathsf{M}_n}\\
 &\leq \lim_{n\to\infty}\frac{n\log_2 n\log 2}{n\log n\log 2}+\lim_{n\to\infty}\frac{n\log_2 n\log 2}{n\log n\log 2}=0
\end{align*}
which follows by b) of Theorem \ref{T_SLLN} together with the Diamond-Vaaler trimmed law, $\log2\left(  \sum_{k=1}^{n}\mathsf{a}%
_{k}-\mathsf{M}_{n}\right)  /(n\log n)\rightarrow1$ a.e.
and finally by a) of Corollary \ref{cor: Ex_PtwGrowth1}.
\end{proof}%

\vspace{0.4cm}%

\begin{proof}
[\textbf{Proof of Theorem \ref{thm: exponents of a}}]
\textbf{a) }
We have that $\int (\mathsf{a}^{\prime})^{\gamma}\,\mathrm{d}\lambda<\infty$ and the statement follows by the ergodic theorem.

\vspace{0.4cm}%

\noindent\textbf{b) }
We may apply \cite[Theorem 1.7 \& erratum]{KS_birkhoff}.
That Property $\mathfrak{C}$ is fulfilled with the bounded variation norm $\|\cdot\|_{BV}$ is a standard result. For Property $\mathfrak{D}$, we notice that
$\|\mathsf{a}\cdot \mathbbm{1}_{\{\mathsf{a}\leq \ell\}}\|_{BV}\leq 2\ell$ and $\|\mathbbm{1}_{\{\mathsf{a}\leq \ell\}}\|_{BV}\leq 2$ implying that this property is fulfilled.

In order to calculate the norming sequence $(d_n)$ we notice that
\begin{align*}
 \mu_{\mathfrak{G}}\left( \left(\mathsf{a}^{\prime}\right)^{\gamma}>n\right)
 &=\mu_{\mathfrak{G}}\left( \mathsf{a}^{\prime}>n^{1/\gamma}\right)\\
 &\sim \frac{1}{\log 2\, n^{1/\gamma} \log n^{1/\gamma}}
 =\frac{\gamma}{\log 2\, n^{1/\gamma} \log n}=\frac{L(n)}{n^{1/\gamma}},
\end{align*}
where $L(n)=\gamma/(\log 2 \log n)$ is a slowly varying function.

Using then \cite[Theorem 1.7 \& erratum]{KS_birkhoff} we obtain that \eqref{eq: strong law gamma>1} holds for $(b_n)$ fulfilling $b_n=o(n)$ and $\lim_{n\to\infty}b_n\log_2 n=\infty$ and for $(d_n)$ fulfilling
\begin{align*}
 d_n\sim \frac{1/\gamma}{1-1/\gamma} n^{\gamma} b_n^{1-\gamma}\left( L^{-\gamma}\right)^{\#}\left(\left(\frac{n}{b_n}\right)^\gamma\right),
\end{align*}
where $\ell^\#$ denotes the de Bruijn conjugate of a slowly varying function $\ell$, see e.g.\ \cite{BGT} for a precise definition.
In our case $\left( L^{-\gamma}\right)^{\#}(n)=\left( (\log n)^\gamma (\log 2)^{\gamma}/\gamma^{\gamma}\right)^{\#}\sim \gamma^{\gamma}/\left((\log n)^\gamma (\log 2)^{\gamma}\right)$.
Hence,
\begin{align*}
 d_n\sim \frac{\gamma^{\gamma}}{(\gamma-1)(\log 2)^{\gamma}} n^{\gamma} b_n^{1-\gamma} \frac{1}{\left(\log \left(n/b_n\right)^\gamma\right)^{\gamma}}
 \sim \frac{1}{(\gamma-1)(\log 2)^{\gamma}}\cdot \frac{n^{\gamma}b_n^{1-\gamma}}{(\log n)^{\gamma}},
\end{align*}
where the last assymptotic follows from the assumption $b_n=o(n^{1-\epsilon})$.
\end{proof}

\section{Proofs of the results on distributional convergence}

We are now ready for the proofs of our distributional convergence results.

\begin{proof}[\textbf{Proof of Theorem \ref{thm: exponents of a in prob}}]
In all cases we only need to check convergence in law w.r.t.\ $\mu_{\mathfrak{G}}$.
\vspace{0.4cm}%

 \noindent\textbf{a) }
 This follows directly from Theorem \ref{thm: exponents of a}.

\vspace{0.4cm}%

 \noindent\textbf{b) }
 This statement follows directly by \cite{A2}. We use the expression for $L'(N)$ from \eqref{Eq_PrimeDigitWR}
which is a slowly varying function.
Since for $b_n=n\log_2 n/\log 2$ we have $nL(b_n)\sim b_n$ the statement follows.

\vspace{0.4cm}%

\noindent\textbf{c) }
This follows from \cite[Theorem 1.8]{KS_mean}. The conditions on the system and the asymptotic of the norming sequence $(d_n)$ we have already considered in the proof of Theorem \ref{thm: exponents of a}.
\end{proof}

\vspace{0.4cm}%

\begin{proof}[\textbf{Proof of Theorem \protect\ref{T_NewPoisson}}]
In each of the three statements it suffices to prove distributional
convergence under the invariant measure $\mu _{\mathfrak{G}}$ (see
Propositions 3.1 and 5.1 in \cite{ZHit}).

\noindent\textbf{a)} For $A\in \mathcal{B}_{I}$ with $\lambda (A)>0$, the \emph{%
(first) hitting time} function of $A$ under the Gauss map $S$, $\varphi
_{A}:I\rightarrow \overline{\mathbb{N}}:=\{1,2,\ldots ,\infty \}$ is given
by $\varphi _{A}(x):=\inf \{n\geq 1:S^{n}x\in A\}$, which is finite a.e.\ on $%
I$. Define $S_{A}x:=S^{\varphi _{A}(x)}x$ for a.e.\ $x\in I$, which gives the
\emph{first entrance map} $S_{A}:I\rightarrow A$. Letting $A_{l}^{\prime
}:=\{\mathsf{a}^{\prime }\geq l\}$, $l\geq 1$, we see that $\varphi
_{l}^{\prime }=\varphi _{A_{l}^{\prime }}$ and, more generally, $\varphi
_{l,i}^{\prime }=\varphi _{A_{l}^{\prime }}\circ S_{A_{l}^{\prime }}^{i-1}$
for $i\geq 1$.\ It is clear that $A_{l}^{\prime }$ is $\xi $-measurable, and
according to Lemma \ref{L_PrimeDigitTail}, $\mu _{\mathfrak{G}}\left(
A_{l}^{\prime }\right) \sim (\log 2\cdot l\log l)^{-1}$ as $l\rightarrow
\infty $. Therefore, Theorem 10.2.a) of \cite{ZHit} immediately implies
statement a).

\vspace{0.4cm}

\noindent\textbf{b)} This is a straightforward consequence of Theorem 10.2.b) in
\cite{ZHit}, because $\{\mathsf{a}^{\prime }\geq l/\vartheta
\}=A_{\left\lfloor l/\vartheta \right\rfloor }^{\prime }$ is $\xi $%
-measurable and (\ref{Eq_PrimeDigitTail}) entails $\mu _{\mathfrak{G}}\left(
\mathsf{a}^{\prime }\geq l/\vartheta \right) \sim \vartheta \,\mu _{%
\mathfrak{G}}\left( \mathsf{a}^{\prime }\geq l\right) $ as $l\rightarrow
\infty $.

\vspace{0.4cm}

\noindent\textbf{c)} Let $\mathbb{P}(j):=\{\mathrm{p}\in \mathbb{P}:\mathrm{p}\equiv
j $ (mod $m$)$\}$, then Dirichlet's PNT for primes in residue classes (e.g.
Theorem 4.4.4 of \cite{Ja}) asserts that for each $j$ relatively prime to $m$%
,%
\begin{equation*}
\#\left( \mathbb{P}(j)\cap \{2,\ldots ,N\}\right) \sim \frac{1}{\phi (m)}%
\frac{N}{\log N}\quad \text{as }N\rightarrow \infty \text{.}
\end{equation*}%
Via an easy argument parallel to the proof of (\ref{Eq_PrimeDigitTail}),
this shows that
\begin{equation*}
\mu _{\mathfrak{G}}\left( A_{l}^{\prime }\cap \left\{ \mathsf{a}^{\prime
}\equiv j\text{ (mod }m\text{)}\right\} \right) \sim \frac{1}{\phi (m)\log 2}%
\cdot \frac{1}{l\log l}\text{ \quad as }l\rightarrow \infty \text{,}
\end{equation*}%
and hence $\mu _{\mathfrak{G}}\left( A_{l}^{\prime }(j)\right) \sim \mu _{%
\mathfrak{G}}\left( A_{l}^{\prime }\right) /\phi (m)$ with $A_{l}^{\prime
}(j):=A_{l}^{\prime }\cap \left\{ \mathsf{a}^{\prime }\equiv j\text{ (mod }m%
\text{)}\right\} $ a $\xi $-measurable set. Another direct application of
Theorem 10.2.b) in \cite{ZHit} then completes the proof of our theorem.
\end{proof}

\vspace{0.4cm}%

The result thus established essentially contains (\ref{Eq_StrgDistrCgeMax}).

\begin{proof}[\textbf{Proof of Theorem \protect\ref{T_ForPhilipp}}]
Theorem \ref{T_NewPoisson} a) contains the statement that $\mu _{\mathfrak{G}%
}\left( A_{l}^{\prime }\right) \varphi _{l,1}^{\prime }$ converges to a
standard exponential law. Using the natural duality $\{\mathsf{M}%
_{n}^{\prime }<l\}=\{\varphi _{l,1}^{\prime }\geq n\}$ this is easily seen
to imply (\ref{Eq_StrgDistrCgeMax}).
\end{proof}

\vspace{0.4cm}%

\begin{proof}[\textbf{Proof of Theorem \ref{thm: counting CLT}.}]
 The result follows directly by \cite[Theorem 3]{KS} by considering the sets $A_n=\{\mathsf{a}_n\in \mathbb{    P}\cap \Gamma_n\}$.
 The only thing to check is that
 \begin{equation}\label{eq: var inf}
  \sum_{n=1}^{\infty}\lambda(A_n)\cdot \lambda(A_n^c)=\infty.
 \end{equation}

 We note that $\sum_{n=1}^{\infty}\lambda(A_n)\cdot \lambda(A_n^c)\geq \sum_{n=1}^{\infty}\lambda(A_n)$ and thus for \eqref{en: clt 1} \eqref{eq: var inf} follows from the proof of Theorem \ref{T_PtwGrowth1} a). \eqref{en: clt 2} corresponds to \cite[Theorem 5A]{KS} and for \eqref{en: clt 3} \eqref{eq: var inf} follows from the proof of Theorem \ref{T_PtwGrowth1} c).
\end{proof}

\vspace{0.4cm}%

\vspace{0.4cm}%

\end{document}